\documentclass[12pt,a4paper]{article}
\usepackage{amsmath,amsthm,amsfonts,amssymb}

 \newcommand{\0}{\bf{0}}
 \newcommand{\bea}{\begin{eqnarray}}

\usepackage{graphicx,latexsym}
\usepackage{lscape}

\newtheorem{thm}{Theorem}[section]
 \newtheorem{cor}[thm]{Corollary}
 \newtheorem{lem}[thm]{Lemma}
 
\newtheorem{ex}[thm]{Example}
 \newtheorem{defn}[thm]{Definition}

\newtheorem{discu}{Discussion:}

\newtheorem{conje}{Conjecture:}

\begin{document}
\begin{center}
\vspace{1cm}
 {\bf \large  Comparison results for  Proper Double Splittings of Rectangular Matrices}
 \vskip 1cm
  {\bf K. Appi Reddy, T. Kurmayya}\\
                        Department of Mathematics\\
                        National Institute of Technology Warangal\\
                                Warangal - 506 004, India. \\
\end{center}

\begin{abstract}
    In this article, we consider two proper double splittings satisfying certain conditions, of a semi-monotone rectangular matrix $A$ and derive new comparison results for the spectral radii of the corresponding iteration matrices. These comparison results are useful to analyse the rate of convergence of the iterative methods (formulated from the double splittings) for solving rectangular linear system $Ax=b.$
\end{abstract}

\vspace{1cm} {\bf Keywords:}\, \,  Double splittings; semi-monotone matrix; spectral radius; Moore-Penrose inverse; Group inverse. \, \,

\vspace{.2cm} {\bf AMS Subject Classification:}\,\,\,\, 15A09, 65F15.\\
\thispagestyle{empty}

\thispagestyle{empty}

\newpage
\section{Introduction}

 Consider the following   linear system,
\begin{equation}
\label{I1.1} Ax=b,
\end{equation}
where $A\in \mathbb{R}^{n\times n}$ is a nonsingular  matrix, $b\in \mathbb{R}^{n\times 1}$ is a given vector and
$x\in \mathbb{R}^{n\times 1}$ is an unknown vector.
In order to solve (\ref{I1.1}), iterative methods of the form
\begin{align}
\label{I1.2}x^{i+1}=Hx^{i}+c,  ~~~~~~~~i=1,2,3...
\end{align}
are often employed. The iterative formula (\ref{I1.2}) is
obtained by splitting $A$ into the form $A=U-V,$ where $U$ is nonsingular and then  setting
$H=U^{-1}V$ and $c=U^{-1}b$. Such a splitting is called a single splitting (see, \cite{wn}) of $A$  and the matrix $H$ is called an iteration matrix \cite{vg}. It is well known (see chapter 7, \cite{bp4}) that the iterative method (\ref{I1.2}) converges to the unique solution of (\ref{I1.1}) (irrespective of the choice of initial vector $x^{\circ}$) if and only if  $\rho(H)<1$, where $\rho(H)$ denotes the spectral radius of $H$, viz., the maximum of the moduli of the eigenvalues of $H$. Note that standard iterative methods like the Jacobi, Gauass-Seidel and successive over-relaxation methods  arise from different choices of real square matrices  $U$ and $V$. A decomposition $A=U-V$ of $A\in \mathbb{R}^{n\times n}$ is called a  regular splitting if $U^{-1}$ exists, $U^{-1} \geq 0$ and $V \geq 0$, where the matrix $B \geq 0$ means all the entries of $B$ are nonnegative. The notion of regular splitting was proposed by Varga \cite{vg} and it was shown that $\rho(H)<1$ if and only if $A$ is monotone. Here,  matrix $A$  monotone \cite{cz} means $A^{-1}$  exists and $A^{-1}\geq 0.$ A decomposition $A=U-V$ of $A\in \mathbb{R}^{n\times n}$  is called a weak regular splitting if  $U^{-1} \geq 0$ and $U^{-1}V \geq 0.$ This was proposed by Ortega and Rheinboldt \cite{or} and again it was shown that $\rho(H)<1$ if and only if $A$ is monotone.
These results show the importance of monotone matrices and  the spectral radius $\rho(H)$ of an iteration matrix,  in the study of convergence of the iterative methods of form (\ref{I1.2}). It is well known that  the convergence of the iterative method (\ref{I1.2}) is faster whenever $\rho(H)$  is smaller and $\rho(H)< 1.$ This leads to the problem of comparison between the spectral radii of the iteration matrices of corresponding iterative methods which are derived from two different splittings $A=U_{1}-V_{1}$ and $A=U_{2}-V_{2}$ of the same matrix $A.$  Results related to this problem are called comparison results for  splittings of matrices. So far, various  comparison theorems for different kinds of single splittings of  matrices have been derived by several authors. For details of these results one could refer to (\cite{besz} to \cite{bp4},  \cite{le}, \cite{ys2},  \cite{vg}, \cite{wn1}  and  \cite{wn2}).

Berman and Plemmons \cite{bp2} then extended the notion of splitting to rectangular matrices and called it as a proper splitting. A decomposition $A=U-V$ of $A\in \mathbb{R}^{m\times n}$ is called a  proper splitting if  $\mathcal{R}(A)=\mathcal{R}(U)$ and $\mathcal{N}(A)=\mathcal{N}(U)$, where
$\mathcal{R}(A)$ and $\mathcal{N}(A)$ denote the range space of $A$ and the null space of $A$, respectively.  Analogous to the invertible case, with such a splitting one associates an iterative sequence $x^{i+1}=Hx^{i}+c,$ where (this time) $H=U^{\dagger}V$ (again) called iteration matrix, $c=U^{\dagger}b$ and $U^{\dagger}$ denotes the Moore-Penreose inverse of $U$ (see next section for definition). The initial vector $x^{\circ},$ however cannot be chosen arbitrarily; it must not belong to   $\mathcal{N}(V).$ Once again it is well known that this sequence converges to $A^{\dagger}b,$ the least square solution of minimum norm, of the system $Ax=b$ (irrespective of the initial vector $x^{\circ}$) if and only if $\rho(H)< 1.$ For details, refer to \cite{bp4}.

Recently, Jena et al. \cite{ljdmsp} extended the notion of regular and weak regular splittings to rectangular matrices and the respective definitions are given next. A decomposition  $A=U-V$ of $A\in \mathbb{R}^{m\times n}$ is called a  proper regular splitting if it is proper splitting such that
$U^{\dagger} \geq 0$ and $V \geq 0.$   It  is called proper weak regular splitting if it is proper splitting such that $U^{\dagger} \geq 0$ and $U^{\dagger}V \geq 0.$ Note that Berman and Plemmons \cite{bp2} proved a convergence theorem for these splittings without specifying the types of matrix decomposition. A matrix $A\in \mathbb{R}^{m\times n}$ is called semi-monotone if $A^{\dagger}\geq 0.$ The authors of \cite{ljdmsp} have considered proper regular splitting of semi monotone matrix $A$ and obtained some comparison results.

 Now, we turn our focus on to  the comparison results for  double splittings that are available in the literature. A decomposition   $A=P-R+S,$ where $P$ is nonsingular, is called a double splitting of $A\in \mathbb{R}^{n\times n}.$ This notion was introduced by $Wo\acute{z}nicki$ \cite{wn}. With such a  splitting,  the following iterative scheme was formulated for  solving (\ref{I1.1}):
\begin{align}
\label{I1.3} x ^{i+1} = P^{-1}Rx^{i}- P^{-1}Sx^{i-1} + P^{-l}b, ~~~~~~~~~~~~i = 1 , 2 , 3. . .
 \end{align}
 Following the idea of Golub and Varga {\cite{gb} }, $Wo\acute{z}nicki$ wrote equation (\ref{I1.3}) in the following equivalent
form: \[
 \begin{pmatrix} x^{i+1} \\
                      x^{i}

                      \end{pmatrix}
                      =\begin{pmatrix} P^{-1}R & -P^{-1}S\\
                     I & 0
                      \end{pmatrix}
                      \begin{pmatrix} x^{i}\\
                      x^{i-1}
                      \end{pmatrix}
                      +\begin{pmatrix} P^{-1}b\\
                      0
                      \end{pmatrix},
                      \]
                      where $I$ is the identity matrix. Then, it was shown that the iterative method (\ref{I1.3}) converges to the unique solution of (\ref{I1.1}) for all initial
vectors $x^{0}$ , $x^{1}$  if and only if the spectral radius of the iteration matrix
\begin{align*}
 W&=\begin{pmatrix} P^{-1}R & -P^{-1}S\\
                     I & 0
                       \end{pmatrix}
\end{align*}
is less than one, that is  $p(W) < 1$.

Based on this idea, in recent years, several comparison theorems have been proved for double splittings of matrices. We briefly review few of them here. First, let us recall the definitions of regular and weak regular double splittings.
A decomposition $A=P-R+S$ is called  regular double splitting if  $P^{-1}\geq 0$, $R\geq 0$ and $-S\geq 0$; it is called
 weak  regular double splitting if  $P^{-1}\geq 0$, $P^{-1}R\geq 0$ and $-P^{-1}S\geq 0$. Shen and Huang \cite{ss} have considered  regular and weak regular double splittings of a monotone matrix or Hermitian positive definite matrix and obtained some comparison theorems. Miao and Zheng \cite{mizh} have obtained comparison theorem for the spectral radii of matrices arising from double splitting of different monotone matrices. Song and Song \cite{sjsy} have studied convergence and comparison theorems for nonnegative double splittings of a real square nonsingular matrices. Li and Wu \cite{li} have obtained some comparison theorems for double splittings of a matrix. Jena et al. \cite{ljdmsp} and  Mishra \cite{deb} have introduced the notions of double proper regular splittings and double proper weak regular  splittings and derived some comparison theorems. Recently,
 Alekha kumar and Mushra \cite{bm} have considered proper nonnegative double splittings of nonnegative matrix and derived certain comparison theorems.

In this article we generalize the comparison results of  Shen and Huang \cite{ss} from square nonsingular matrices to rectangular matrices and from classical inverses to Moore-Penrose inverses. Infact, we consider two double splittings  $A=P_{1}-R_{1}+S_{1}$ and $A=P_{2}-R_{2}+S_{2}$ of a semi-monotone matrix  $A\in \mathbb{R}^{m\times n}$ and derive two comparison theorems for the spectral radii of the corresponding iteration matrices. In section 2, we introduce notations and preliminary results. We present main results in section 3.

\section{Notations,  Definitions and Preliminaries}
 In this section, we fix notations and collect basic definitions and preliminary results which will be used in the sequel. Let $\mathbb{R}^{m\times n}$ denote the set of all  real matrices with $m$ rows and $n$ columns. For $A\in \mathbb{R}^{m\times n},$ the transpose of $A$ is denoted by $A^{t}$; and
 the matrix $X\in \mathbb{R}^{n\times m}$ satisfying $AXA=A$, $XAX=X$, $(AX)^{t}=AX$ and $(XA)^{t}=XA$ is called the Moore-Penrose inverse of $A.$ It always exists and unique, and is denoted by $A^{\dagger}$. If $A$ is invertible then $A^{\dagger}=A^{-1}$.
Let  $L$ and $M$ be complementary subspaces of a real Euclidean space $\mathbb{R}^{n}$. Then the projection of $\mathbb{R}^{n}$ on $L$ along $M$
is denoted by $P_{L,M}$. If, in addition, $L$ and $M$ are orthogonal then it is called an orthogonal projection and it is  denoted simply by $P_{L}$.  The following well known properties (see, \cite{bg}) of
$A^{\dagger}$,    will be  used in this manuscript: $\mathcal{R}(A^{t})=\mathcal{R}(A^{\dagger})$, $\mathcal{N}(A^{t})=\mathcal{N}(A^{\dagger})$, $AA^{\dagger}=P_{\mathcal{R}(A)}$, $A^{\dagger}A=P_{\mathcal{R}(A^{t})}$.  In particular, if $x\in \mathcal{R}(A^{t})$ then $x=A^{\dagger}Ax$.

A matrix $A\in \mathbb{R}^{m\times n}$ is nonnegative, if all the entries of $A$ are nonnegative, this is denoted $A\geq 0.$ The same notation
and nomenclature are also used for vectors. For $A, B\in \mathbb{R}^{m\times n}$, we write $B\geq A$ if $B-A\geq 0.$

We now present some results connecting nonnegativity of a matrix and its spectral radius.
\begin{lem} (Theorem 2.1.11, \cite {bp4})\label{com}
Let $A\in\mathbb {R}^{n\times n}$ and $A\geq 0$. Then $\alpha x\leq Ax$, $x\geq 0\Rightarrow \alpha \leq \rho(A)$ and $Ax\leq \beta x$, $x> 0 \Rightarrow \rho(A)\leq \beta$.
\end{lem}
\begin{thm}(Theorem 3.16, \cite{vg})\label{Theorem1}
  Let $B\in\mathbb {R}^{n\times n}$ and $B\geq 0$. Then $\rho(B)< 1$ if and only if $(I-B)^{-1}$ exists and
  $(I-B)^{-1}=\sum_{k=0}^{\infty} B^{k}\geq 0$.
   \end{thm}
   The next theorem is a part of the Perron-Frobenius theorem.
   \begin{thm}(Theorem 2.20, \cite{vg})\label{Theorem2}
Let $A\in\mathbb {R}^{n\times n}$ and $A\geq 0$ . Then
\newline
$(i)$ $A$ has a nonnegative real eigenvalue equal to the spectral radius.
\newline
$(ii)$ There exists a nonnegative real eigenvector for its  spectral radius.
   \end{thm}
   \begin{lem}(Lemma 2.2, \cite{ss})\label {spectral}
 Let $A=\begin{pmatrix} B & C\\
                     I & 0
                      \end{pmatrix} \geq 0$ and $\rho(B+C)< 1$. Then, $\rho(A)< 1$.
\end{lem}
As we mentioned in the introduction, a decomposition $A=U-V$ of $A\in \mathbb{R}^{m\times n}$ is called a  proper splitting if  $\mathcal{R}(A)=\mathcal{R}(U)$ and $\mathcal{N}(A)=\mathcal{N}(U).$ The next two results are on proper splittings.
\begin{thm}(Theorem 1, \cite {bp2})
If $A=U-V$ is a proper splitting of $A\in \mathbb{R}^{m\times n}$, then $AA^{\dagger}=UU^{\dagger}$ and
$A^{\dagger}A=U^{\dagger}U.$
\end{thm}

\begin{thm} ( Theorem 3, \cite{bp2}) \label{eq}
Let $A=U-V$ be a  proper splitting of $A\in \mathbb{R}^{m\times n}$ such that $U^{\dagger}\geq 0$ and $U^{\dagger}V\geq 0$. Then the following are equivalent:\\
(i)   $A^{\dagger}\geq 0$.\\
(ii)  $A^{\dagger}V\geq 0$.\\
(iii) $\rho(U^{\dagger}V)<1.$
\end{thm}
 Note that, a proper splitting $A=U-V$  of $A\in \mathbb{R}^{m\times n}$, satisfying the conditions $U^{\dagger}\geq 0$ and $U^{\dagger}V\geq 0$ is named as proper weak regular splitting by Jena et al. \cite{ljdmsp}.

   We now turn to results on double splittings. For $A\in \mathbb{R}^{m\times n}$, a decomposition $A=P-R+S$ is called a $double~ splitting$ of $A$.  A double  splitting $A=P-R+S$ of $A\in \mathbb{R}^{m\times n}$ is called a  $ proper~ double ~splitting$  if $\mathcal{R}(A)=\mathcal{R}(P)$
 and $\mathcal{N}(A)=\mathcal{N}(P)$. Again, consider the following   rectangular linear system
\begin{equation}
\label{I1.4} Ax=b,
\end{equation}
where $A\in \mathbb{R}^{m\times n}$ (this time $A$ need not be nonsingular), $b\in \mathbb{R}^{m\times 1}$ is a given vector and
$x\in \mathbb{R}^{n\times 1}$ is an unknown vector.
 Similar to the nonsingular case, if we  use proper double splitting $A=P-R+S$ to solve (\ref{I1.4}), it leads to the following iterative scheme:
 \begin{align}
    \label{I1.5}                    x^{k+1}=P^{\dagger}Rx^{k}-P^{\dagger}Sx^{k-1}+P^{\dagger}b, ~\text{where} ~k= 1, 2, ...
                        \end{align}
   Motivated by $Wo\acute{z}nicki's$ {\cite{wn} } idea ,  equation (\ref{I1.5}) can be written as
 \[
 \begin{pmatrix} x^{k+1}\\
                      x^{k}
                      \end{pmatrix}
                      =\begin{pmatrix} P^{\dagger}R & -P^{\dagger}S\\
                     I & 0
                      \end{pmatrix}
                      \begin{pmatrix} x^{k}\\
                      x^{k-1}
                      \end{pmatrix}
                      +\begin{pmatrix} P^{\dagger}b\\
                      0
                      \end{pmatrix}.
                      \]
       If we denote,  $X^{k+1}=\begin{pmatrix} x^{k+1}\\
                      x^{k}

                      \end{pmatrix},$
          $W= \begin{pmatrix} P^{\dagger}R & -P^{\dagger}S\\
                     I & 0
                      \end{pmatrix},$
                      $X^{k}=\begin{pmatrix} x^{k}\\
                      x^{k-1}
                     \end{pmatrix}$\\
                     and $B=\begin{pmatrix} P^{\dagger}b\\
                      0
                      \end{pmatrix},$ then we get
                      \begin{equation}
  \label{I1.6}         X^{k+1}=WX^{k}+B,      k=1,2...
          \end{equation}
          Then, it can be shown that the iterative method (\ref{I1.6}) converges to to the unique least square solution of minimum norm, of (\ref{I1.4}) if and only if  $\rho(W)<1.$

Next, we introduce some subclasses of   proper double   splittings.
\begin{defn}
  Let $A\in \mathbb{R}^{m\times n}$. A proper double  splitting $A=P-R+S$ is called\\
(i) regular proper double splitting if $P^{\dagger}\geq 0$, $R\geq 0$ and $-S\geq 0.$\\
(ii) weak regular proper double splitting  if  $P^{\dagger}\geq 0$, $P^{\dagger}R\geq 0$ and $-P^{\dagger}S\geq 0.$
\end{defn}
Note that the authors of \cite{ljdmsp}  called regular proper double splittings and weak regular proper double splittings as double proper regular splittings and double proper weak regular splittings, respectively. However, we feel that the present usage is more appropriate and hence we  continue the same nomenclature throughout this manuscript.

The next result gives the relation between the spectral radius of the iteration matrices associated with a single splitting and a double splitting.
\begin{thm}(Theorem 4.3, \cite{deb} )
Let $A=P-R+S$ be a weak  regular proper double splitting of $A\in \mathbb{R}^{m\times n}.$ Then $\rho(W)< 1$ if and only if $\rho(U^{\dagger}V)< 1,$
where $U=P$ and $V=R-S.$
\end{thm}

We conclude this section with a  convergence theorem for a  proper double splitting of a monotone matrix.
\begin{thm} (Theorem 3.6, \cite{ljdmsp} )\label{con}
Let $A\in \mathbb{R}^{m\times n}$ such that $A^{\dagger}\geq 0$. Let $A=P-R+S$ be a  weak regular proper double splitting. Then, $\rho(W)< 1$.
\end{thm}
\section{Main Results}
In this section, the main results of this article are presented. These results extend the results of Shen and Huang \cite{ss} from square nonsingular matrices to rectangular matrices and from classical inverses to Moore-Penrose inverses.

Let $A\in \mathbb{R}^{m\times n}$. Let $A=P_{1}-R_{1}+S_{1}=P_{2}-R_{2}+S_{2}$ be two double proper splittings of $A$. Set
$W_{1}=\begin{pmatrix} P_{1}^{\dagger}R_{1} & -P_{1}^{\dagger}S_{1}\\
                     I & 0
                      \end{pmatrix}$ and
                      $W_{2}=\begin{pmatrix} P_{2}^{\dagger}R_{2} & -P_{2}^{\dagger}S_{2}\\
                     I & 0
                      \end{pmatrix}$.

  The next result gives the comparison between $\rho(W_{1})$ and $\rho(W_{2})$. As mentioned earlier, this  comparison is useful to analyse the rate of convergence of the iterative methods formulated from these double splittings, for solving linear system $Ax=b$.
\begin{thm}\label{dcr1}
Let $A\in \mathbb{R}^{m\times n}$ be such that $A^{\dagger}\geq 0$. Let $A=P_{1}-R_{1}+S_{1}$ be a   regular proper double splitting such that $P_{1}P_{1}^{\dagger}\geq 0$ and let $A=P_{2}-R_{2}+S_{2}$ be a   weak regular proper double splitting. If $P_{1}^{\dagger}\geq P_{2}^{\dagger}$ and any one of the following conditions,\\
(i) $P_{1}^{\dagger}R_{1}\geq P_{2}^{\dagger}R_{2}$\\
(ii) $P_{1}^{\dagger}S_{1}\geq P_{2}^{\dagger}S_{2}$ \\ holds, then $\rho(W_{1})\leq \rho(W_{2})< 1 $.
\end{thm}
\begin{proof}
Since $A=P_{1}-R_{1}+S_{1}$ is a   regular proper double splitting of $A,$ by Theorem  \ref{con}, we get $\rho(W_{1})< 1.$ Similarly, $\rho(W_{2})< 1.$
It  remains to show that $\rho(W_{1})\leq \rho(W_{2}).$

Assume that $\rho(W_{1})=0$. Then the conclusion follows, obviously. So, without loss of generality assume that
 $\rho(W_{1})\neq 0$. Since  $A=P_{1}-R_{1}+S_{1}$ is a  regular proper double  splitting, we have  $W_{1}=\begin{pmatrix} P_{1}^{\dagger}R_{1} & -P_{1}^{\dagger}S_{1}\\
                     I & 0
                      \end{pmatrix}\geq 0.$ Then,  by the Perron-Frobenius theorem,
there exists a vector $x=\begin{pmatrix} x_{1}\\
                           x_{2}
                           \end{pmatrix}\in \mathbb{R}^{2n},$ $x\geq 0$ and $x\neq 0$ such that $W_{1}x=\rho(W_{1})x.$
               This implies that
               \begin{align}
    \label{dcr1.1}           P_{1}^{\dagger}R_{1}x_{1}-P_{1}^{\dagger}S_{1}x_{2}&=\rho(W_{1})x_{1}.     \\
    \label{dcr1.2}            x_{1}&=\rho(W_{1})x_{2}.
               \end{align}
     Upon pre multiplying equation (\ref{dcr1.1}) by $P_{1}$   and  using equation (\ref{dcr1.2}), we  get
  \begin{align}
\label{dcr1.3}      [\rho(W_{1})]^{2}P_{1}x_{1}=\rho(W_{1})P_{1}P_{1}^{\dagger}R_{1}x_{1}-P_{1}P_{1}^{\dagger}S_{1}x_{1}.
\end{align}
We have $P_{1}P_{1}^{\dagger}\geq 0,$ $R_{1}\geq 0,$ $-S_{1}\geq 0$ and $x_{1}\geq 0.$
Therefore, by (\ref{dcr1.3}), $[\rho(W_{1})]^{2}P_{1}x_{1}\geq 0.$\\
Now, again from (\ref{dcr1.3}),
       \begin{align*}
       0&=[\rho(W_{1})]^{2}P_{1}x_{1}-\rho(W_{1})P_{1}P_{1}^{\dagger}R_{1}x_{1}+P_{1}P_{1}^{\dagger}S_{1}x_{1}\\&\leq
                                        \rho(W_{1})P_{1}x_{1}-\rho(W_{1})P_{1}P_{1}^{\dagger}R_{1}x_{1}+\rho(W_{1})P_{1}P_{1}^{\dagger}S_{1}x_{1}\nonumber\\
                                        &=\rho(W_{1})[P_{1}x_{1}-P_{1}P_{1}^{\dagger}(R_{1}-S_{1})x_{1}]\\
                                        &=\rho(W_{1})[P_{1}x_{1}-R_{1}x_{1}+S_{1}x_{1}]\\
                                        &=\rho(W_{1})Ax_{1},
                                        \end{align*}
  where we have used the facts that   $0< \rho(W_{1})< 1$ and    $\mathcal{R}(R_{1}-S_{1})\subseteq \mathcal{R}(P_{1}).$
                                  This proves that $Ax_{1}\geq 0.$\\
    Also, by using equations   (\ref{dcr1.1}) and (\ref{dcr1.2}), we get
               \begin{align*}
      W_{2}x-&\rho(W_{1})x=  \begin{pmatrix}P_{2}^{\dagger}R_{2}x_{1}-P_{2}^{\dagger}S_{2}x_{2}-\rho(W_{1})x_{1} \\
                           x_{1}-\rho(W_{1})x_{2}
                           \end{pmatrix}\\
                      &=\begin{pmatrix}(P_{2}^{\dagger}R_{2}-P_{1}^{\dagger}R_{1})x_{1}+\frac{1}{\rho(W_{1})}(P_{1}^{\dagger}S_{1}-P_{2}^{\dagger}S_{2})x_{1} \\
                           0
                           \end{pmatrix}\\
                           &=\begin{pmatrix} \nabla\\
                           0
                           \end{pmatrix},
                      \end{align*}
           where $\nabla = (P_{2}^{\dagger}R_{2}-P_{1}^{\dagger}R_{1})x_{1}+\frac{1}{\rho(W_{1})}(P_{1}^{\dagger}S_{1}-P_{2}^{\dagger}S_{2})x_{1}.$\\
 {\bf Case(i)} Let us assume that $P_{1}^{\dagger}R_{1}\geq P_{2}^{\dagger}R_{2}.$ Since $0<\rho(W_{1})< 1,$ we get
             $(P_{2}^{\dagger}R_{2}-P_{1}^{\dagger}R_{1})x_{1}\geq \frac{1}{\rho(W_{1})}(P_{2}^{\dagger}R_{2}-P_{1}^{\dagger}R_{1})x_{1}$. Then
             \begin{align}
       \nabla=&(P_{2}^{\dagger}R_{2}-P_{1}^{\dagger}R_{1})x_{1}+ \frac{1}{\rho(W_{1})}(P_{1}^{\dagger}S_{1}-P_{2}^{\dagger}S_{2})x_{1}\nonumber\\&\geq
       \frac{1}{\rho(W_{1})}(P_{2}^{\dagger}R_{2}-P_{1}^{\dagger}R_{1})x_{1}+ \frac{1}{\rho(W_{1})}(P_{1}^{\dagger}S_{1}-P_{2}^{\dagger}S_{2})x_{1}\nonumber\\
       &=\frac{1}{\rho(W_{1})}[(P_{2}^{\dagger}(R_{2}-S_{2})x_{1}- P_{1}^{\dagger}(R_{1}-S_{1})x_{1}]\nonumber\\
       &=\frac{1}{\rho(W_{1})}[P_{2}^{\dagger}P_{2}- P_{2}^{\dagger}A-P_{1}^{\dagger}P_{1}+ P_{1}^{\dagger}A]x_{1}\nonumber\\
       &=\frac{1}{\rho(W_{1})}(P_{1}^{\dagger}-P_{2}^{\dagger})Ax_{1},
       \end{align}
      where we have used the fact that $P_{1}^{\dagger}P_{1}=P_{2}^{\dagger}P_{2}.$  Since $Ax_{1}\geq 0$ and $P_{1}^{\dagger}\geq P_{2}^{\dagger},$ from the above inequality, we get $\nabla\geq 0.$ Then, $W_{2}x-\rho(W_{1})x=\begin{pmatrix} \nabla\\
                           0
                           \end{pmatrix} \geq 0.$
                           This implies that   $\rho(W_{1})x\leq W_{2}x$. So, by Lemma \ref{com}, $\rho(W_{1})\leq \rho(W_{2})$. This proves that
                           $\rho(W_{1})\leq \rho(W_{2})< 1.$\\
    {\bf Case(ii)} Assume that $P_{1}^{\dagger}S_{1}\geq P_{2}^{\dagger}S_{2}$. Since
 $0<\rho(W_{1})< 1$ and $Ax_{1}\geq 0,$ agian we get
  \begin{align*}
 \nabla=  &(P_{2}^{\dagger}R_{2}-P_{1}^{\dagger}R_{1})x_{1}+ \frac{1}{\rho(W_{1})}(P_{1}^{\dagger}S_{1}-P_{2}^{\dagger}S_{2})x_{1}\\&\geq
       (P_{2}^{\dagger}R_{2}-P_{1}^{\dagger}R_{1})x_{1}+ (P_{1}^{\dagger}S_{1}-P_{2}^{\dagger}S_{2})x_{1}\\
       &=(P_{1}^{\dagger}-P_{2}^{\dagger})Ax_{1}\geq 0.
    \end{align*}
    This implies that $W_{2}x-\rho(W_{1})x=\begin{pmatrix} \nabla\\
                           0
                           \end{pmatrix} \geq 0.$\\
    So, again by Lemma \ref{com}, we get $\rho(W_{1})\leq \rho(W_{2})$. This proves that
                           $\rho(W_{1})\leq \rho(W_{2})< 1.$\\
\end{proof}
The following example shows that the converse of   Theorem \ref{dcr1} is not true.
\begin{ex}
Let $A=\begin{pmatrix} 3 & -2 & 0\\
                      -1 & 1 & 0
                      \end{pmatrix}$. Let $P_{1}=\begin{pmatrix} 5 & -1 & 0\\
                      0 & 1 & 0
                      \end{pmatrix},$
                     \newline $R_{1}= \begin{pmatrix} 1 & 0 & 0\\
                      0 & 0 & 0
                      \end{pmatrix} ,$
                      $S_{1}= \begin{pmatrix} -1 & -1 & 0\\
                      -1 & 0 & 0
                      \end{pmatrix},$
    $P_{2}=\begin{pmatrix} 3 & 0 & 0\\
                      0 & 2& 0
                      \end{pmatrix},$
                     $R_{2}= \begin{pmatrix} 0 & 1 & 0\\
                      0 & 1 & 0
                      \end{pmatrix} $   and
                      $S_{2}= \begin{pmatrix} 0 & -1 & 0\\
                      -1 & 0 & 0
                      \end{pmatrix} $.
  Then $P_{1}^{\dagger}=\frac{1}{5}\begin{pmatrix} 1 & 1\\
                      0 & 5\\
                      0 & 0
                      \end{pmatrix},$
                      $P_{1}^{\dagger}R_{1}=\frac{1}{5}\begin{pmatrix} 1 & 0 & 0\\
                      0 & 0 & 0\\
                      0 & 0 & 0
                      \end{pmatrix},$
                      $P_{1}^{\dagger}S_{1}=\frac{1}{5}\begin{pmatrix} -2 & -1 & 0\\
                      -5 & 0 & 0\\
                      0 & 0 & 0
                      \end{pmatrix},$
   $P_{2}^{\dagger}=\frac{1}{6}\begin{pmatrix} 2 & 0\\
                      0 & 3\\
                      0 & 0
                      \end{pmatrix},$
                      $P_{2}^{\dagger}R_{2}=\frac{1}{6}\begin{pmatrix} 0 & 2 & 0\\
                      0 & 3 & 0\\
                      0 & 0 & 0
                      \end{pmatrix}$ and
                      $P_{2}^{\dagger}S_{2}=\frac{1}{6}\begin{pmatrix} 0 & -2 & 0\\
                      -3 & 0 & 0\\
                      0 & 0 & 0
                      \end{pmatrix}$.
            It is easy to verify that  $A=P_{1}-R_{1}+S_{1}$ is a   regular proper double splitting and
$A=P_{2}-R_{2}+S_{2}$ is a  weak regular proper double  splitting.
Also,  $0.9079=\rho(W_{1})\leq \rho(W_{2})=0.9158 < 1$.
 However, the conditions $P_{1}^{\dagger}\geq P_{2}^{\dagger} ,$ $P_{1}^{\dagger}R_{1}\geq P_{2}^{\dagger}R_{2}$ and
 $P_{1}^{\dagger}S_{1}\geq P_{2}^{\dagger}S_{2}$ do not hold.
\end{ex}
\begin{cor} (Theorem 3.1, \cite {ss}) \label{cor1}
 Let $A^{-1}\geq 0$. Let $A=P_{1}-R_{1}+S_{1}$ be a  regular double splitting and
$A=P_{2}-R_{2}+S_{2}$ be a weak  regular double splitting. If $P_{1}^{-1}\geq P_{2}^{-1}$ and any one of the
following conditions,\\
(i) $P_{1}^{-1}R_{1}\geq P_{2}^{-1}R_{2}$\\
(ii) $P_{1}^{-1}S_{1}\geq P_{2}^{-1}S_{2}$ \newline holds, then $\rho(W_{1})\leq \rho(W_{2})< 1, $ where
$W_{1}=\begin{pmatrix} P_{1}^{-1}R_{1} & -P_{1}^{-1}S_{1}\\
                     I & 0
                      \end{pmatrix}$ and
                      \newline$W_{2}=\begin{pmatrix} P_{2}^{-1}R_{2} & -P_{2}^{-1}S_{2}\\
                     I & 0
                      \end{pmatrix}$.
 \end{cor}
 \begin{cor}\label{cor2}
 Let $A^{-1}\geq 0$. Let $A=P_{1}-R_{1}+S_{1}$ be a  regular double splitting and
$A=P_{2}-R_{2}+S_{2}$ be a weak  regular double splitting. If $P_{1}^{-1}\geq P_{2}^{-1}$ and $R_{1}\geq R_{2}$ hold,
then $\rho(W_{1})\leq \rho(W_{2})< 1 $.
 \end{cor}
    The conclusion of Theorem \ref{dcr1} can also be achieved by replacing  a regular proper  double splitting $A=P_{1}-R_{1}+S_{1}$  with a weak regular proper double splitting; and a weak regular proper double splitting $A=P_{2}-R_{2}+S_{2}$ with  a regular proper double splitting, in Theorem
 \ref{dcr1}. The following is the exct statement of this result.
\begin{thm} \label{mr1}
Let $A\in \mathbb{R}^{m\times n}$ such that $e=(1,1,...,1)^{t}\in \mathcal{R}(A)$ and $A^{\dagger}\geq 0$. Let $A=P_{1}-R_{1}+S_{1}$ be a  weak regular proper double splitting and let $A=P_{2}-R_{2}+S_{2}$ be a  regular proper double splitting such that   $P_{2}^{\dagger}$ has no zero row and $P_{2}P_{2}^{\dagger}\geq 0$. If $P_{1}^{\dagger}\geq P_{2}^{\dagger}$ and any one of the
following conditions,\\
(i) $P_{1}^{\dagger}R_{1}\geq P_{2}^{\dagger}R_{2}$\\
(ii) $P_{1}^{\dagger}S_{1}\geq P_{2}^{\dagger}S_{2}$ \\ holds, then $\rho(W_{1})\leq \rho(W_{2})< 1 $.
\end{thm}
\begin{proof}
Since $A=P_{1}-R_{1}+S_{1}$ is a weak regular proper double splitting of $A,$ by Theorem  \ref{con}, we get $\rho(W_{1})< 1.$ Similarly, $\rho(W_{2})< 1.$ It  remains to show that $\rho(W_{1})\leq \rho(W_{2}).$

 Let $J$ be  an $m\times n$ matrix in which  each entry is equal to 1. For given $\epsilon >0,$ set $A_{\epsilon}=A-\epsilon J,$
 $R_{1}(\epsilon)=R_{1}+\frac{1}{2}\epsilon J$, $S_{1}(\epsilon)=S_{1}-\frac{1}{2}\epsilon J$,\linebreak
     $R_{2}(\epsilon)=R_{2}+\frac{1}{2}\epsilon J,$ $S_{2}(\epsilon)=S_{2}-\frac{1}{2}\epsilon J,$
  $W_{1}(\epsilon)=\begin{pmatrix} P_{1}^{\dagger}R_{1}(\epsilon) & -P_{1}^{\dagger}S_{1}(\epsilon)\\
                     I & 0
                      \end{pmatrix}$ and
                      $W_{2}(\epsilon)=\begin{pmatrix} P_{2}^{\dagger}R_{2}(\epsilon) & -P_{2}^{\dagger}S_{2}(\epsilon)\\
                     I & 0
                      \end{pmatrix}$.
                       We have, $e=(1,1,...,1)^{t}\in \mathcal{R}(A)$. So, there exists a
matrix $B\in \mathbb{R}^{n\times n}$ such that $J=AB$. Then $A_{\epsilon}=A-\epsilon J=(A-\epsilon AB)=(A-\epsilon AA^{\dagger}AB)=(A-\epsilon AA^{\dagger} J)=A(I-\epsilon A^{\dagger}J).$
Now, choose the above $\epsilon $  such that
   $\rho(\epsilon A^{\dagger}J)< 1$ and $\mathcal{N}(A_{\epsilon})=\mathcal{N}(A)$. Since $\rho(\epsilon A^{\dagger}J)< 1,$ $I-\epsilon A^{\dagger}J$ is invertible and hence  $\mathcal{R}(A_{\epsilon})=\mathcal{R}(A).$ Then $A_{\epsilon}=A-\epsilon J$ becomes a proper splitting and thus we can conclude that $A_{\epsilon}=P_{1}-R_{1}(\epsilon)+S_{1}(\epsilon)$ is  a   weak regular proper double splitting and
$A_{\epsilon}=P_{2}-R_{2}(\epsilon)+S_{2}(\epsilon)$ is a regular proper double  splitting.

For the same $\epsilon,$ define  $X=(I-\epsilon A^{\dagger}J)^{-1}A^{\dagger},$ we shall prove that $X$ is the Moore-Penrose inverse of $A_{\epsilon}.$
Let $x\in \mathcal{R}(A_{\epsilon}^{t}).$ Then
 \begin{align*}
 XA_{\epsilon}x&=(I-\epsilon A^{\dagger}J)^{-1}A^{\dagger}(A-\epsilon AA^{\dagger} J)x\\
 &=(I-\epsilon A^{\dagger}J)^{-1}(A^{\dagger}Ax-\epsilon A^{\dagger}AA^{\dagger}Jx)\\
 &=(I-\epsilon A^{\dagger}J)^{-1}(x-\epsilon A^{\dagger}Jx)\\
 &=x
  \end{align*}
 and for  $y\in \mathcal{N}(A_{\epsilon}^{t}),$ we get
 \begin{align*}
 Xy=(I-\epsilon A^{\dagger}J)^{-1}A^{\dagger}y=0
 \end{align*}
 Hence, by the definition, $A^{\dagger}_{\epsilon}=X=(I-\epsilon A^{\dagger}J)^{-1}A^{\dagger}.$
 Also, \newline $A^{\dagger}_{\epsilon}=(I+\epsilon A^{\dagger}J+\epsilon (A^{\dagger}J)^{2}+...)A^{\dagger}\geq 0.$
Then $\rho(P_{2}^{\dagger}(R_{2}(\epsilon)-S_{2}(\epsilon)))< 1.$ So, by Lemma \ref{spectral}, $\rho(W_{2}(\epsilon))< 1.$

 Clearly, $P_{2}^{\dagger}R_{2}(\epsilon)> 0$ and $-P_{2}^{\dagger}S_{2}(\epsilon)> 0.$
 So, $W_{2}(\epsilon)\geq 0$. Then, by the Perron-Frobenius theorem, there exists a vector
$x(\epsilon)=\begin{pmatrix} x_{1}(\epsilon)\\
                           x_{2}(\epsilon)
                           \end{pmatrix}\in \mathbb{R}^{2n},$ $x(\epsilon)\geq 0$ and $x(\epsilon)\neq 0$ such that
                           $W_{2}(\epsilon)x(\epsilon)=\rho(W_{2}(\epsilon))x(\epsilon).$ This implies,
               \begin{align}
  \label{2.1}P_{2}^{\dagger}R_{2}(\epsilon)x_{1}(\epsilon)-P_{2}^{\dagger}S_{2}(\epsilon)x_{2}(\epsilon)&=\rho(W_{2}(\epsilon))x_{1}(\epsilon)\\     \label{2.2}  x_{1}(\epsilon)&=\rho(W_{2}(\epsilon))x_{2}(\epsilon).
               \end{align}
    If $\rho (W _{2}(\epsilon))=0$ then from equations (\ref{2.1}) and  (\ref{2.2}), $x(\epsilon)=0.$ This is a  contradiction. So,
    $0< \rho (W _{2}(\epsilon))< 1.$
Then by using equations   (\ref{2.1}) and (\ref{2.2}), as in the proof of the Theorem \ref{dcr1}, we can show that $\rho(W_{2}(\epsilon))A_{\epsilon}x_{1}(\epsilon)\geq 0.$ This implies that $A_{\epsilon}x_{1}(\epsilon)\geq0.$ Also, from equations (\ref{2.1}) and (\ref{2.2}), we get
\begin{align*}
          &W_{1}(\epsilon)x(\epsilon)-\rho(W_{2}(\epsilon))x(\epsilon)\\
          &=\begin{pmatrix}P_{1}^{\dagger}R_{1}(\epsilon)x_{1}(\epsilon)-P_{1}^{\dagger}S_{1}(\epsilon)x_{2}(\epsilon)-\rho(W_{2}(\epsilon))x_{1}(\epsilon) \\
                           x_{1}(\epsilon)-\rho(W_{2}(\epsilon))x_{2}(\epsilon)
                           \end{pmatrix}\\
                           &=\begin{pmatrix} (P_{1}^{\dagger}R_{1}(\epsilon)-P_{2}^{\dagger}R_{2}(\epsilon))x_{1}(\epsilon)+ \frac{1}{\rho(W_{2}(\epsilon))}(P_{2}^{\dagger}S_{2}(\epsilon)-P_{1}^{\dagger}S_{1}(\epsilon))x_{1}(\epsilon)\\
                           0
                           \end{pmatrix}\\
                           &=\begin{pmatrix} \nabla\\
                           0
                           \end{pmatrix},
                            \end{align*}
       where $\nabla= (P_{1}^{\dagger}R_{1}(\epsilon)-P_{2}^{\dagger}R_{2}(\epsilon))x_{1}(\epsilon)+ \frac{1}{\rho(W_{2}(\epsilon))}(P_{2}^{\dagger}S_{2}(\epsilon)-P_{1}^{\dagger}S_{1}(\epsilon))x_{1}(\epsilon).$\\
{\bf Case$(i)$} Assume that $P_{1}^{\dagger}R_{1}\geq P_{2}^{\dagger}R_{2}$. Since  $0< \rho (W _{2}(\epsilon))< 1,$ we get that
$(P_{1}^{\dagger}R_{1}(\epsilon)-P_{2}^{\dagger}R_{2}(\epsilon))x_{1}(\epsilon)\leq \frac{1}{\rho(W_{2}(\epsilon))}(P_{1}^{\dagger}R_{1}(\epsilon)-P_{2}^{\dagger}R_{2}(\epsilon))x_{1}(\epsilon).$ \text{Therefore,}
 \begin{align*}
\nabla &\leq
       \frac{1}{\rho(W_{2}(\epsilon))}(P_{1}^{\dagger}R_{1}(\epsilon)-P_{2}^{\dagger}R_{2}(\epsilon))x_{1}(\epsilon)+ \frac{1}{\rho(W_{2}(\epsilon))}(P_{2}^{\dagger}S_{2}(\epsilon)-P_{1}^{\dagger}S_{1}(\epsilon))x_{1}(\epsilon)\\
       &=\frac{1}{\rho(W_{2}(\epsilon))}[(P_{1}^{\dagger}(R_{1}(\epsilon)-S_{1}(\epsilon))x_{1}(\epsilon)- P_{2}^{\dagger}(R_{2}(\epsilon)-S_{2}(\epsilon))x_{1}(\epsilon)]\\
       &=\frac{1}{\rho(W_{2}(\epsilon))}[P_{1}^{\dagger}P_{1}- P_{1}^{\dagger}A_{\epsilon}-P_{2}^{\dagger}P_{2}+ P_{2}^{\dagger}A_{\epsilon}]x_{1}(\epsilon)\\
       &=\frac{1}{\rho(W_{2}(\epsilon))}(P_{2}^{\dagger}-P_{1}^{\dagger})A_{\epsilon}x_{1}(\epsilon)
       \end{align*}
    where we have used the fact that $P_{1}^{\dagger}P_{1}=P_{2}^{\dagger}P_{2}.$
    Since $A_{\epsilon}x_{1}(\epsilon)\geq 0$ and $P_{1}^{\dagger}\geq P_{2}^{\dagger},$ we get that $\nabla\leq 0.$
     Thus, $ W_{1}(\epsilon)x(\epsilon)-\rho(W_{2}(\epsilon))x(\epsilon)=\begin{pmatrix} \nabla\\
                           0
                           \end{pmatrix}\leq 0.$
                           This implies,  $ W_{1}(\epsilon)x(\epsilon)\leq \rho(W_{2}(\epsilon))x(\epsilon).$ So, by Lemma \ref{com}, $\rho(W_{1}(\epsilon))\leq \rho(W_{2}(\epsilon)).$\\
Now,  from the continuity of eigenvalues, we have $$\rho(W_{1})=\lim_{\epsilon\to\0} \rho(W_{1}(\epsilon))\leq \lim_{\epsilon\to\0} \rho(W_{2}(\epsilon))=\rho(W_{2}).$$
   {\bf Case$(ii)$} Assume that $P_{1}^{\dagger}S_{1}\geq P_{2}^{\dagger}S_{2}.$ We have $\rho(\epsilon A^{\dagger}J)< 1.$
Choose the above  $\epsilon$  small enough such that
 $$P_{1}^{\dagger}S_{1} -P_{2}^{\dagger}S_{2}\geq \frac{\epsilon}{2}(P_{1}^{\dagger} -P_{2}^{\dagger})J.$$
Since, $P_{1}^{\dagger}S_{1}(\epsilon)\geq P_{2}^{\dagger}S_{2}(\epsilon),$  $A^{\dagger}_{\epsilon}\geq 0$ and  $0<\rho(W_{2})< 1,$
 we get \begin{align*}
 \nabla & \leq (P_{1}^{\dagger}R_{1}(\epsilon)-P_{2}^{\dagger}R_{2}(\epsilon))x_{1}(\epsilon)+ (P_{2}^{\dagger}S_{2}(\epsilon)-P_{1}^{\dagger}S_{1}(\epsilon))x_{1}(\epsilon)\\
 &=(P_{2}^{\dagger}-P_{1}^{\dagger})A_{\epsilon}x_{1}(\epsilon)\leq 0.
 \end{align*}
 This implies that $ W_{1}(\epsilon)x(\epsilon)-\rho(W_{2}(\epsilon))x(\epsilon)=\begin{pmatrix} \nabla\\
                           0
                           \end{pmatrix} \leq 0.$\\
                          So,   $ W_{1}(\epsilon)x(\epsilon)\leq \rho(W_{2}(\epsilon))x(\epsilon).$ Then, by Lemma \ref{com}, $\rho(W_{1}(\epsilon))\leq \rho(W_{2}(\epsilon)).$
        Similar to the proof of case$(i)$, this implies that $\rho(W_{1})\leq \rho(W_{2})$.
       \end{proof}
 The following example illustrates Theorem \ref{mr1}.
\begin{ex}
Let $A=\begin{pmatrix} 1 & 0 & 1\\
                      0 & 1 & 0\\
                      \end{pmatrix}$
                      then $A^{\dagger}=\frac{1}{2}\begin{pmatrix} 1 & 0 \\
                      0 & 2 \\
                      1 & 0
                      \end{pmatrix}\geq 0.$\\
 Set $P_{1}=\begin{pmatrix} 3 & 0 & 3\\
                      0 & 3 & 0\\
                      \end{pmatrix},$
                     $R_{1}= \begin{pmatrix} 2 & 0 & 2\\
                      0 & 1 & 0\\
                       \end{pmatrix} $   and
                      $S_{1}= \begin{pmatrix} 0 & 0 & 0\\
                      0 & -1 & 0\\
                     \end{pmatrix} $.\\
    $P_{2}=\begin{pmatrix} 4 & 0 & 4\\
                      0 & 4 & 0\\
                     \end{pmatrix},$
                     $R_{2}= \begin{pmatrix} 2 & 0 & 2\\
                      0 & 0 & 0\\
                      \end{pmatrix} $   and
                      $S_{2}= \begin{pmatrix} -1 & 0 & -1\\
                      0 & -3 & 0\\
                      \end{pmatrix} $.\\
  Then $P_{1}^{\dagger}=\frac{1}{6}\begin{pmatrix} 1 & 0 \\
                      0 & 2 \\
                      1 & 0
                      \end{pmatrix},$
                      $P_{1}^{\dagger}R_{1}=\frac{1}{6}\begin{pmatrix} 2 & 0 & 2\\
                      0 & 2 & 0\\
                      2 & 0 & 2
                      \end{pmatrix}$ and
                      $P_{1}^{\dagger}S_{1}=\frac{1}{6}\begin{pmatrix} 0 & 0 & 0\\
                      0 & -2 & 0\\
                      0 & 0 & 0
                      \end{pmatrix}$.\\
   $P_{2}^{\dagger}=\frac{1}{8}\begin{pmatrix} 1 & 0 \\
                      0 & 2 \\
                      1 & 0
                      \end{pmatrix},$
                      $P_{2}^{\dagger}R_{2}=\frac{1}{8}\begin{pmatrix} 2 & 0 & 2\\
                      0 & 0 & 0\\
                      2 & 0 & 2
                      \end{pmatrix}$ and
                      $P_{2}^{\dagger}S_{2}=\frac{1}{8}\begin{pmatrix} -1 & 0 & -1\\
                      0 & -6 & 0\\
                      -1 & 0 & -1
                      \end{pmatrix}$.\\
            Note that  $A=P_{1}-R_{1}+S_{1}$ is a   weak regular proper double splitting and $A=P_{2}-R_{2}+S_{2}$  is a regular proper  double splitting.
            Also, $e\in \mathcal{R}(A)$, $P_{2}^{\dagger}$ has no zero row and $P_{2}P_{2}^{\dagger}\geq 0$.
We can verify that $P_{1}^{\dagger}\geq P_{2}^{\dagger},$  and $P_{1}^{\dagger}R_{1}\geq P_{2}^{\dagger}R_{2}.$
Hence $0.7676=\rho(W_{1})\leq \rho(W_{2})=0.6660 < 1$.
\end{ex}
The following result is an obvious consequence of Theorem 3.5. 
\begin{cor} (Theorem 3.2, \cite {ss}) \label{cor4}
 Let $A^{-1}\geq 0$. Let $A=P_{1}-R_{1}+S_{1}$ be a weak regular double splitting and
$A=P_{2}-R_{2}+S_{2}$ be a   regular double splitting. If $P_{1}^{-1}\geq P_{2}^{-1}$ and any one of the
following conditions,\\
(i) $P_{1}^{-1}R_{1}\geq P_{2}^{-1}R_{2}$\\
(ii) $P_{1}^{-1}S_{1}\geq P_{2}^{-1}S_{2}$ \\ holds, then $\rho(W_{1})\leq \rho(W_{2})< 1 $.
\end{cor}
The next result proof is similar to the proof of Theorem $3.1$. Thus we skip the proof.
 \begin{thm}\label{mr3}
   Let $A\in \mathbb{R}^{m\times n}$ and $A^{\dagger}\geq 0$. Let $A=P_{1}-R_{1}+S_{1}$ be a  weak  regular proper double splitting and
$A=P_{2}-R_{2}+S_{2}$ be a weak regular proper double splitting. If $P_{1}^{\dagger}A\geq P_{2}^{\dagger}A$ and any of the
following conditions,\\
(i) $P_{1}^{\dagger}R_{1}\geq P_{2}^{\dagger}R_{2}$\\
(ii) $P_{1}^{\dagger}S_{1}\geq P_{2}^{\dagger}S_{2}$ \\ holds, then $\rho(W_{1})\leq \rho(W_{2})< 1 $.
    \end{thm}


\newpage
 \begin{thebibliography}{10}
 \bibitem{bm}
 A. K. Baliarsingh and D. Mishra,  \emph{Comparison results for proper nonnegative splittings of matrices}, Results in  Mathematics. DOI 10.1007/s00025-015-0504-9.
  \bibitem{bg}
 A.Ben-Israel  and T.N.E. Greville,    \emph{Generalized Inverses:Theory and Applications},  2nd edition,  Springer Verlag,
New York, 2003.
\bibitem{besz}
M. Benzi and D.B. Szyld, \emph{Existence and uniqueness of splittings for stationary iterartive methods with applications to alternating methods},
Numer. Math, {\bf 76(3)} (1997), 309-321.
\bibitem{bp2}
A. Berman and R.J. Plemmons,  \emph{Cones and iterative methods for best least squres solutions of linear systems}, SIAM J. Numer. Anal, \textbf{11}(1974) 145-154.
\bibitem{bp3}
A. Berman and R.J. Plemmons, \emph{Inverses of Nonnegative matrices}, Linear and Multilinar
Algebra, \textbf{2}(1974) 161-172.
\bibitem{bp4}
A. Berman and R.J. Plemmons, \emph{Nonnegative Matrices in the Mathematical Sciences}, Classics in Applied Mathematics, SIAM,
1994.
\bibitem{cz}
L. Collatz, \emph{Functional Analysis and Numerical Mathematics},  Academic, New York, 1966.
\bibitem{le}
L. Elsner, A. Frommer, R. Nabben, H. Schneider and D.B. Szyld, \emph{ Conditons for strict inequality in comparisions of spectral radii of splittings of different matrices}.  Linear Algebra Appl, {\bf 363} (2003), 65-80.
\bibitem{gb}
 C.H. Golub and R.S. Varga,  \emph{Chebyshev semi-iterative methods, successive overrrelaxation iterative methods,
and second order Richardson iterative methods}, I, Numer. Math, \textbf{3}(1961) 147-168.
\bibitem{ljdmsp}
 L. Jena,  D. Mishra  and S. Pani,  \emph{Convergence and comparison theorems for single
and double decompositions of rectangular matrices}, Calcolo, \textbf{51}(2014) 141-149.
\bibitem{li}
C.X. Li  and S.L. Wu,  \emph{Some New Comparison Theorems for Double Splittings of Matrices},
Appl. Math. Inf. Sci, \textbf{8(5)}(2014) 2523-2526.
\bibitem{mizh}
S. X. Miao  and B. Zheng,  \emph{A note on double splittings of different monotone matrices}, Calcolo, \textbf{46}(2009) 261-266.
\bibitem{deb}
D. Mishra,  \emph{Nonnegatve splittings for rectangular matrices},
Computers and Mathematics with applications, \textbf{67}(2014) 136-144.
\bibitem{or}
J.M. Ortega and W.C. Rheinboldt, \emph{Monotone iterations for nonlinear equations with application to Gauss-Seidel methods}, SIAM J. Numer. Anal, \textbf{4}(1967) 171-190.
\bibitem{ss}
S.Q. Shen  and T.Z. Huang,  \emph{Convergence and Comparision Theorems for Double Splittings of Matrices},
Computers and Mathematics with Applications, \textbf{51}(2006) 1751-1760.
\bibitem{ys2}
Y. Song,  \emph{Comparison thoerems for  splittings of matrices},  Numer.Math, {\bf 92} (2002), 563-591.
\bibitem{sjsy}
J. Song  and Y. Song,  \emph{Convergence for nonnegative double splittings of matrices}, Calcolo, \textbf{48}(2011) 245-260.
\bibitem{vg}
R. S. Varga, \emph{Matrix iterative analysis}, volume 27 of Springer Series in Computational Mathematics. Springer-Verlag,
Berlin, 2000, expanded edition.
\bibitem{wn}
 $Z.I. Wo\acute{z}nicki$, \emph{ Estimation of the optimum relaxation factors in partial factorization iterative methods}, SIAM
J. Matrix Anal. Appl, \textbf{14}(1993) 59-73.
\bibitem{wn1}
 $Z.I. Wo\acute{z}nicki$, \emph{Basic comparison theorems for weak and weaker matrix splittings}, Electronic Journal of Linear Algebra, \textbf{8}(2001) 53-59.
 \bibitem{wn2}
 $Z.I. Wo\acute{z}nicki$,  \emph{Nonnegative  splitting theorey}, Japan Journal of Industrial and Applied mathematics, \textbf{11(2)}(1994) 289-342.
\end{thebibliography}
\end{document}